\documentclass[a4paper,british]{article}
\usepackage[T1]{fontenc}
\usepackage[latin9]{inputenc}
\usepackage{amsmath}
\usepackage{amsthm}
\usepackage{amssymb}
\usepackage[numbers]{natbib}
\usepackage{hyperref}
\usepackage{changes} 
\usepackage{subeqnarray}
\definechangesauthor[name={Mathias}, color=blue]{mat}
\definechangesauthor[name={Markus}, color=blue]{mb}

\makeatletter

\theoremstyle{plain}
\newtheorem{theorem}{\protect\theoremname}
\theoremstyle{plain}
\newtheorem{proposition}[theorem]{\protect\propositionname}
\theoremstyle{plain}
\newtheorem{corollary}[theorem]{\protect\corollaryname}
\theoremstyle{plain}
\newtheorem{lemma}[theorem]{\protect\lemmaname}

\usepackage{a4wide}
\usepackage{bbm}

\allowdisplaybreaks

\makeatother

\usepackage{babel}
\providecommand{\corollaryname}{Corollary}
\providecommand{\lemmaname}{Lemma}
\providecommand{\propositionname}{Proposition}
\providecommand{\theoremname}{Theorem}
\newcommand{\KLEINO}{{\scriptstyle{\mathcal{O}}}}

\begin{document}
\global\long\def\phi{\varphi}%
 
\global\long\def\epsilon{\varepsilon}%
 
\global\long\def\eps{\varepsilon}%
 
\global\long\def\theta{\vartheta}%
\global\long\def\E{\mathbb{E}}%
\global\long\def\Var{\operatorname{Var}}%
 
\global\long\def\Cov{\operatorname{Cov}}%
 
\global\long\def\N{\mathbb{N}}%
 
\global\long\def\Z{\mathbb{Z}}%
 
\global\long\def\R{\mathbb{R}}%
 
\global\long\def\F{\mathcal{F}}%
 
\global\long\def\le{\leqslant}%
 
\global\long\def\ge{\geqslant}%
 
\global\long\def\MT{\ensuremath{\clubsuit}}%
 
\global\long\def\1{\mathbbm{1}}%
 
\global\long\def\d{\mathrm{d}}%
 
\global\long\def\P{\mathbb{P}}%
 
\global\long\def\subset{\subseteq}%
 
\global\long\def\supset{\supseteq}%
 
\global\long\def\argmin{\operatorname*{arg\, min}}%

\global\long\def\rank{\operatorname{rank}}%
 
\global\long\def\diag{\operatorname{diag}}%
 
\global\long\def\arginf{\operatorname*{arg\, inf}}%
 
\global\long\def\bull{{\scriptstyle \bullet}}%
 
\global\long\def\supp{\operatorname{supp}}%
 
\global\long\def\sgn{\operatorname{sign}}%
 
\global\long\def\tr{\operatorname{tr}}%
 
\global\long\def\spn{\operatorname{span}}%
 
\global\long\def\ran{\operatorname{ran}}%
 
\global\long\def\Id{\operatorname{Id}}%
\global\long\def\Bigtimes{\bigtimes}%

\title{On central limit theorems for power variations of the solution to the stochastic heat equation}
\author{Markus Bibinger$^a$ and Mathias Trabs$^b$}
\date{}

\maketitle

\vspace*{-.4cm}
\emph{\small $^a$Fachbereich12 Mathematik und Informatik, Philipps-Universität Marburg, bibinger@uni-marburg.de}

\emph{\small $^b$Fachbereich Mathematik, Universität Hamburg, mathias.trabs@uni-hamburg.de}

\begin{abstract}
We consider the stochastic heat equation whose solution is observed
discretely in space and time. An asymptotic analysis of power variations is presented including the proof of a central limit theorem. It generalizes the theory from Bibinger and Trabs \cite{BibingerTrabs2017} in several directions.
\end{abstract}
\vspace{.25cm}


\section{Introduction and main result}
\label{sec:1}
Stochastic partial differential equations (SPDEs) do not only provide key models in modern probability theory, but also become increasingly popular
in applications, for instance, in neurobiology or mathematical finance. Consequently, statistical methods are required to calibrate SPDE models from given observations. However, in the statistical literature on SPDEs, see \cite{Cialenco2018} for a recent review, there are still basic questions which are not yet settled.

A natural problem is parameter estimation based on discrete observations of a solution of an SPDE which was first studied in \cite{Markussen2003} and which has very recently attracted considerable interest. Applying similar methods the three related independent works \cite{cialenco,BibingerTrabs2017,chong} study parabolic SPDEs including the stochastic heat equation, consider high-frequency observations in time, construct estimators using power variations of time-increments of the solution and prove central limit theorems. As we shall see below, the marginal solution process along time at a fixed spatial point is not a (semi-)martingale such that the well-established high-frequency theory for stochastic processes from \cite{JP} cannot be (directly) applied. In view of this difficulty, different techniques are required to prove central limit theorems. Interestingly, the proof strategies in \cite{cialenco,BibingerTrabs2017,chong} are quite different. Cialenco and Huang \cite{cialenco} consider the realised fourth power variation for the stochastic heat equation with both an unbounded spatial domain $\mathcal{D}=\R$, or a bounded spatial domain $\mathcal{D}=[0,\pi]$. In the first setting they apply the central limit theorem by Breuer and Major \cite{BreuerMajor1983} for stationary Gaussian sequences with sufficient decay of the correlations. For $\mathcal{D}=[0,\pi]$, they use Malliavin calculus instead and the fourth moment theorem from \cite{NualartOritz-Latorre2008}. Also in case of a bounded domain $\mathcal{D}=[0,1]$, with Dirichlet boundary conditions, Bibinger and Trabs \cite{BibingerTrabs2017} study the normalized discrete quadratic variation and establish its asymptotic normality building upon a theorem by Peligrad and Utev \cite{PeligradUtev1997} for triangular arrays which satisfy a covariance inequality related to $\rho$-mixing. Finally, Chong \cite{chong} has proved (stable) central limit theorems for power variations in the case $\mathcal{D}=\R$, based on a non-obvious martingale approximation in combination with the theory from \cite{JP}. The strategy of proofs by \cite{BibingerTrabs2017} and \cite{chong} do not directly rely on a purely Gaussian model and can be transferred to more general settings. While \cite{BibingerTrabs2017} considers further nonparametric inference on a time-varying deterministic volatility, \cite{chong} already provides a proof beyond the Gaussian framework including stochastic volatility.

This note presents a concise analysis which transfers the asymptotic theory from \cite{BibingerTrabs2017} to an unbounded spatial domain $\mathcal{D}=\R$, and from the normalized discrete quadratic variation to general power variations. Contrarily to \cite{BibingerTrabs2017}, we do not start with the illustration of a solution as an infinite-dimensional SDE but exploit the explicit representation of the solution with the heat kernel thanks to the continuous spectrum of the Laplace operator on the whole real line. We stick here to the simplest Gaussian setting to illustrate the main aspects and deviations from the classical theory. Our findings show that the central limit theorem under a $\rho$-mixing type condition used in \cite{BibingerTrabs2017} for the case with a bounded spatial domain can be used likewise for this different model with unbounded spatial domain. We moreover expect that it provides a perspective to prove central limit theorems very generally, although many approximation details, for instance, to address stochastic volatility, remain far from being obvious. We consider the \emph{stochastic heat equation} in one spatial dimension
\begin{equation}
\partial_{t}X_{t}(x)  =\frac{\theta}{2}\partial_{xx}X_{t}(x)+\sigma\dot{W}(t,x),\quad X_{0}(x)=\xi(x),\quad t>0,x\in\mathbb{R},\label{eq:heat}
\end{equation}
for space-time white noise $\dot{W}$, and with parameters $\theta,\sigma>0$, and some initial condition $\xi$ which is independent of $\dot{W}$.
$\dot{W}$ is defined as a centred Gaussian process with covariance structure $\E[\dot{W}(s,x)\dot{W}(t,y)]=\1_{s=t}\1_{x=y}$, and is in terms of a distribution the space-time derivative of a Brownian sheet. Since the Laplace operator on the whole real line does not have a discrete spectrum and we do not have to discuss boundary problems, the asymptotic analysis actually simplifies compared to \cite{BibingerTrabs2017} and allows for more transparent proofs. 

A mild solution of (\ref{eq:heat}) is a random field that admits the representation
\begin{equation}
X_{t}(x)=\int_{\R}G(t,x-y)\xi(y)dy+\int_{0}^{t}\int_{\R}G(t-s,x-y)\sigma\dot{W}(ds,dy)\,,\label{eq:mildSol}
\end{equation}
for $t\ge0,x\in\R$, where the integral is well-defined as the stochastic Walsh integral and with
\[
G(t,x):=\frac{\exp(-x^{2}/(2\theta t))}{\sqrt{2\pi\theta t}}\,.
\]
$G(t,x)$ is the heat kernel, the fundamental solution to the heat equation. Let us refer to \cite[Ch.~2.3.1]{lototsky} for an introduction to the heat equation and SPDEs in general. Suppose we observe this solution on a discrete grid $(t_{i},x_{k})_{i=0,\dots,n;k=1,\dots,m}\subset\R_{+}\times\R$, at equidistant observation times $t_{i}:=i\Delta_{n}$. We consider infill or high-frequency asymptotics where $\Delta_{n}\downarrow0$. For statistical inference on the parameters in \eqref{eq:heat}, the key quantities to study are \emph{power variations}
\[
V_{n}^p(x):=\frac{1}{n}\sum_{i=1}^{n}\bigg|\frac{\Delta_{i}X(x)}{\Delta_n^{1/4}}\bigg|^{p},\qquad\Delta_{i}X(x):=X_{i\Delta_n}(x)-X_{(i-1)\Delta_n}(x)\,,
\]
with $p\in\N$. The normalization of $\Delta_{i}X(x)$ with $\Delta_{n}^{1/4}$ takes into account the (almost) $\frac{1}{4}$-H\"older regularity in time of $X_{t}(x)$, see \cite[Ex.~2.3.5]{lototsky}. By homogeneity in space, statistics to consider for volatility estimation are spatial averages
\begin{equation}
\bar{V}_{n,m}^p:=\frac{1}{m}\sum_{k=1}^{m}V_{n}^p(x_{k})=\frac{1}{nm}\sum_{i=1}^{n}\sum_{k=1}^{m}\bigg|\frac{\Delta_{i}X(x_k)}{\Delta_n^{1/4}}\bigg|^{p}\,.\label{eq:Vnm}
\end{equation}
The main result of this note is a central limit theorem for $\bar{V}^p_{n,m}$ in the double asymptotic regime where $n\to\infty$ and (possibly) $m\to\infty$. An important role in our asymptotic analysis is played by the second-order increment operator $D_{2}(f,s):=f(s)-2f(s-1)+f(s-2)$ for some function $f$, being well defined on $[s-2,s]$. For brevity we assume $\xi=0$, but the result readily extends to sufficiently regular initial conditions which are independent of $\dot W$.
\begin{theorem}\label{thm:clt}
  Consider (\ref{eq:heat}) with $\xi=0$. For $\delta_{m}:=\min_{k=2,\dots,m}|x_{k}-x_{k-1}|$ assume that $\Delta_{n}/\delta_{m}^{2}\to0$ as $n\vee m\to\infty$. Then the power variations from (\ref{eq:Vnm}) with $p\in\N$ satisfy as $n\to\infty$ and $\Delta_n\to 0$
  \begin{align*}
  \sqrt{m\cdot n}\Big(\bar{V}_{n,m}^p-\Big(\frac{2}{\pi\theta}\Big)^{\frac{p}{4}}\sigma^p\mu_p\Big)\overset{d}{\longrightarrow}\mathcal{N}\bigg(0,\Big(\frac{2}{\pi\theta}\Big)^{\frac{p}{2}}\sigma^{2p}\Big((\mu_{2p}-\mu_p^2)+2\sum_{r=2}^{\infty}\rho_p\big(\tfrac12 D_2(\sqrt{\cdot},r)\big)\Big)\bigg)\,,
  \end{align*}
  with $\mu_p=\E[|Z|^p],Z\sim \mathcal{N}(0,1)$, and with $\rho_p(a)=\Cov(|Z_1|^p,|Z_2|^p)$ for $Z_1,Z_2$ jointly normally distributed with expectation 0, variances 1 and correlation $a$.
\end{theorem}
Note the explicit formula $\mu_p=2^{p/2}\Gamma\big(\tfrac{p+1}{2}\big)/\sqrt{\pi}$, also referred to as $(p-1)!!$ for $p$ even. In particular for $p=2$, that is, for the normalized discrete quadratic variation, we have $\mu_2=1$ and the asymptotic variance is 
\[\Big(\Big(\frac{2}{\pi\theta}\Big)^{1/4}\sigma\Big)^4\Big(2+\sum_{r=2}^{\infty} (D_2(\sqrt{\cdot},r))^2\Big)\]
in analogy with Example 2.11 in \cite{chong} and with \cite{BibingerTrabs2017}. This coincides with the variance of the normalized discrete quadratic variation of a fractional Brownian motion with Hurst exponent $1/4$ and scale parameter $(2/(\pi\theta))^{1/4}\sigma$, see also Theorem 6 in \cite{power} and \cite{tudor}.

The above result allows for a growing time horizon $T:=n\Delta_{n}=\KLEINO(n)$ and, more general than in \cite{BibingerTrabs2017}, the number $m$ of spatial observations in the unbounded spatial domain can be larger than the number of observation times $n$. The relevant condition that induces de-correlated observations in space is $\Delta_{n}/\delta_{m}^{2}\to0$, tantamount to a finer observation frequency in time than in space. Based on Theorem~\ref{thm:clt}, one can construct estimators and confidence statements for the parameters $\sigma^{2}$ and $\theta$, if the other one is known, see \cite{cialenco,BibingerTrabs2017,chong}. If no parameter is known apriori, \cite[Sec.\ 5]{BibingerTrabs2017} show that the ``viscosity-adjusted volatility'' $\sigma^2\sqrt{2/\vartheta}$ can be estimated consistently, also noted in \cite[Sec.\ 2.3]{chong}. 
\section{High-frequency asymptotic analysis of power variations}
Our analysis builds upon the following result, whose proof is postponed to Section~\ref{sec:proofs}.
\begin{proposition}
\label{prop:CovInk} For $x,y\in\R$ with $x\neq y$, we
have that
\begin{align*}
\Cov(\Delta_{i}X(x),\Delta_{j}X(x))=& \sqrt{\Delta_{n}} \sqrt{\frac{2}{\pi\theta}}\sigma^2\big(\1_{\{i=j\}}\hspace*{-.04cm}+\hspace*{-.04cm}\tfrac12 D_{2}(\sqrt{\cdot},|i-j|+1)\1_{\{i\ne j\}}\hspace*{-.04cm}+\hspace*{-.04cm}\tfrac12 D_{2}(\sqrt{\cdot},i+j)\big)\;\text{and}\\
|\hspace*{-.05cm}\Cov(\Delta_{i}X(x),\Delta_{j}X(y))|=& \;\mathcal{O}\bigg(\frac{\Delta_{n}}{|x-y|}\Big(\frac{1}{|i-j-1|\vee1}+\1_{\{i=j\}}\Big)\bigg)\,.
\end{align*}
\end{proposition}
The increments thus have non-negligible covariances and $t\mapsto X_{t}(x)$ is not a (semi-)martingale. The terms $D_{2}(\sqrt{\cdot},i+j)$ will turn out to be asymptotically negligible in the variance of the power variations. Since second-order differences $D_{2}(\sqrt{\cdot},\cdot)$ of the square root decay as its second derivative, we observe that $\Cov(\Delta_{i}X(x),\Delta_{j}X(x))=\mathcal{O}(\sqrt{\Delta_{n}}(i-j)^{-3/2})$. This motivates an asymptotic theory exploiting $\rho$-mixing arguments. From the proposition and joint normality of the increments, we readily obtain the expectation and variance of the power variations $V_{n}^p(x)$ at one spatial point $x\in\R$.
\begin{corollary}
\label{cor:expVar}For any $x\in\R$, we have that
\begin{subeqnarray}
\E[V_{n}^p(x)]=\Big(\frac{2}{\pi\theta}\Big)^{\frac{p}{4}}\sigma^p\mu_p+\mathcal{O}(n^{-1})~\text{and}\hspace*{4.5cm}\\
\hspace*{-.5cm}\Var(V_{n}^p(x))=\frac{1}{n}\Big(\frac{2}{\pi\theta}\Big)^{\frac{p}{2}}\sigma^{2p}\Big((\mu_{2p}-\mu_p^2)\hspace*{-.05cm}+\hspace*{-.05cm}2\sum_{r=2}^{\infty}\rho_p\big(\tfrac12 D_2(\sqrt{\cdot},r)\big)\Big)\hspace*{-.05cm}+\hspace*{-.05cm}\KLEINO\Big(\frac{1}{n}\Big)
\end{subeqnarray}
with $\mu_p=\E[|Z|^p],Z\sim  \mathcal{N}(0,1)$, and with $\rho_p(a)=\Cov(|Z_1|^p,|Z_2|^p)$ for $Z_1,Z_2$ jointly centred Gaussian with variances 1 and correlation $a$.
\end{corollary}

\begin{proof}
For $i=j$, Proposition~\ref{prop:CovInk} yields $\Var(\Delta_{i}X(x))=\sqrt{\Delta_{n}}\sigma^{2}\sqrt{2/\pi\theta}(1+\tfrac12 D_{2}(\sqrt{\cdot},2i))$. Since $|D_{2}(\sqrt{\cdot},2i)|\le\tfrac14 (2(i-1))^{-3/2}$, we obtain by a Taylor expansion that 
\begin{equation*}
\E[V_{n}^p(x)]  =\frac{1}{n}\sum_{i=1}^{n}\mu_p\big|\sigma^2\sqrt{2/\pi\theta} \big(1+\tfrac12 D_{2}(\sqrt{\cdot},2i)\big)\big|^{\frac{p}{2}}=\Big(\frac{2}{\pi\theta}\Big)^{\frac{p}{4}}\sigma^p\mu_p+\mathcal{O}\big(n^{-1}\big).
\end{equation*}
Using the joint normality of the increments $(\Delta_{i}X)_{1\le i\le n}$, and writing $\Delta_{i}X=(2\Delta_n/\pi\theta)^{1/4}\sigma \tilde Z_{x,i}$, with a tight sequence $(\tilde Z_{x,i})_{1\le i\le n}$, we deduce for any $x\in\R$ that
\begin{align}
\Var(V_{n}^p(x))  &=\frac{1}{n^{2}}\sum_{i,j=1}^{n}\Cov\bigg(\bigg|\frac{\Delta_{i}X(x)}{\Delta_n^{1/4}}\bigg|^{p},\bigg|\frac{\Delta_{j}X(x)}{\Delta_n^{1/4}}\bigg|^{p}\bigg)\label{eq:cov}\\
 &=\sigma^{2p}\Big(\frac{2}{\pi\theta}\Big)^{\frac{p}{2}}  \Big(\frac{1}{n^{2}}\sum_{i=1}^n\Var(| Z_1|^p)|1+\tfrac12 D_2(\sqrt{\cdot},2i)|^p+\frac{2}{n^2}\sum_{i=1}^{n}\sum_{j=1}^{i-1}\Cov(|\tilde Z_{x,i}|^p,|\tilde Z_{x,j}|^p)\Big)\,.\nonumber
\end{align}
By the above bound, the term with $D_{2}(\sqrt{\cdot},2i)$ is negligible such that $\Var(\tilde Z_{x,i})\approx 1$ up to this negligible term. For the covariance terms, we use Proposition~\ref{prop:CovInk} to obtain
\begin{align*}
\frac{1}{n}\sum_{i=1}^n\sum_{j=1}^{i-1}\Cov\big(|\tilde Z_{x,i}|^p,|\tilde Z_{x,j}|^p\big) 
&=\frac{1}{n}\sum_{i=1}^n\sum_{j=1}^{i-1}\rho_p\Big(\operatorname{corr}\big(\tilde Z_{x,i},\tilde Z_{x,j}\big)\Big)+\KLEINO(1)\\
& =\frac{1}{n}\sum_{i=1}^n\sum_{j=1}^{i-1}\rho_p\Big(\tfrac12 D_2(\sqrt{\cdot},i-j)\Big)+\KLEINO(1)\\
& =\sum_{r=2}^{\infty}\rho_p\Big(\tfrac12 D_2(\sqrt{\cdot},r)\Big)+\KLEINO(1)\,.
\end{align*}
The first equality comes from approximating the variances by one and the second approximation is based on the Hermite expansion of absolute power functions \eqref{hermite} with Hermite rank 2, see also \cite[(A.6)]{power}. The last estimate follows from
\[
  \rho_p\Big(\tfrac12 D_2(\sqrt{\cdot},i-j)\Big)=\mathcal O\Big(D_2(\sqrt{\cdot},i-j)^2\Big)=\mathcal O((i-j)^{-3}),\quad i>j.
\]
\end{proof}
As we can see from the previous proof, the term $(2\sigma^{4}/(\pi\theta))^{p/2}(\mu_{2p}-\mu_p^2)$ in the variance would also appear for independent increments, while the additional term involving $\rho_p$ comes from the non-vanishing covariances. Proposition~\ref{prop:CovInk} moreover implies that the covariance of $V_{n}^p(x)$ and $V_{n}^p(y)$ decreases with a growing distance of the spatial observation points $x$ and $y$. In particular, averaging over all spatial observations in (\ref{eq:Vnm}) reduces the variance by the factor $1/m$, as long as the high-frequency regime in time dominates the spatial resolution. The next corollary determines the asymptotic variance in Theorem~\ref{thm:clt}.
\begin{corollary}
\label{cor:varVnm} Under the conditions of Theorem \ref{thm:clt}, we have that
\[
\Var\big(\bar{V}_{n,m}^p\big)\hspace*{-.05cm}=\hspace*{-.05cm}\frac{1}{mn}\Big(\frac{2}{\pi\theta}\Big)^{\frac{p}{2}}\sigma^{2p}\Big((\mu_{2p}-\mu_p^2)+2\sum_{r=2}^{\infty}\rho_p\big(\tfrac12 D_2(\sqrt{\cdot},r)\big)\Big)\big(1+\KLEINO(1)\big).
\]
\end{corollary}

\begin{proof}
For $U_1,U_2$ bivariate Gaussian with correlation $a$ and variances $\sigma_1^2$ and $\sigma_2^2$, we exploit the inequality
\[\Cov\big(|U_1|^p\,,\,|U_2|^p\big)\le C_p\sigma_1^p\sigma_2^pa^2\,,\]
with some constant $C_p$, which is based on the Hermite expansion \eqref{hermite} and given in Equation (4) of \cite{Guyon1987}, see also Lemma 3.3 of \cite{PAKKANEN}.\\
By this inequality and Proposition~\ref{prop:CovInk} for $x\neq y$, we deduce that
\begin{align*}\frac{1}{m^2}\sum_{k\ne l} \Cov\big(V_{n}^p(x_{k}),V_{n}^p(x_{l})\Big) &=\frac{1}{m^2n^2}\sum_{i,j}\sum_{k\ne l}\Cov\bigg(\bigg|\frac{\Delta_{i}X(x_k)}{\Delta_n^{1/4}}\bigg|^{p},\bigg|\frac{\Delta_{j}X(x_l)}{\Delta_n^{1/4}}\bigg|^{p}\bigg)\\
&=\mathcal{O}\bigg(\frac{1}{m^2n^2}\sum_{i,j}\sum_{k\ne l}\frac{\Delta_n}{|x_k-x_l|^2}\Big(\1_{\{i=j\}}+\frac{1}{|i-j-1|\vee1}\Big)^{2}\bigg).
\end{align*}
With the estimate
\begin{align*}
&\frac{1}{n^{2}}\sum_{i=1}^{n}\frac{\Delta_{n}}{|x_k-x_l|^2}+\frac{1}{n^{2}}\sum_{i=3}^{n}\sum_{j=1}^{i-2}\frac{\Delta_{n}}{|x_k-x_l|^2}\frac{1}{(i-j-1)^2} \\
&\quad \le \frac{\Delta_{n}}{n|x_k-x_l|^2}+\frac{\Delta_{n}}{n^{2}|x_k-x_l|^2}\sum_{i=3}^{n}\sum_{k=1}^{i-2}k^{-2}=\mathcal{O}\Big(\frac{\Delta_{n}}{n|x_k-x_l|^2}\Big)\,,
\end{align*}
we obtain in combination with Corollary~\ref{cor:expVar} that 
\begin{align*}
\Var(\bar{V}^p_{n,m}) & =\frac{1}{m^{2}}\Big(\sum_{k=1}^{m}\Var\big(V_{n}^p(x_{k})\big)+\sum_{k\neq l}\Cov\big(V_{n}^p(x_{k}),V_{n}^p(x_{l})\big)\Big)\\
 & =\frac{1}{mn}\Big(\frac{2}{\pi\theta}\Big)^{\frac{p}{2}}\sigma^{2p}\Big((\mu_{2p}-\mu_p^2)+2\sum_{r=2}^{\infty}\rho_p\big(\tfrac12 D_2(\sqrt{\cdot},r)\big)\Big)+\KLEINO\Big(\frac{1}{mn}\Big)
\end{align*}
under the condition $\Delta_n\delta_m^{-2}\to 0$, where we use that 
\begin{equation}\label{h_spatial}\sum_{k\ne l}\frac{1}{|x_k-x_l|^2}\le 2\delta_m^{-2}\sum_{k=2}^m\sum_{l=1}^{k-1}l^{-2}=\mathcal{O}\big(\log{(m)}\delta_m^{-2}\big)=\mathcal{O}\big(m\delta_m^{-2}\big)\,.\end{equation}
\end{proof}
We turn to the proof of the central limit theorem transferring the strategy from \cite{BibingerTrabs2017} to our model. Define
the triangular array
\[
Z_{n,i}:=\frac{1}{\sqrt{mn}}\sum_{k=1}^{m}\Delta_{n}^{-\frac{p}{4}}\big(|\Delta_{i}X(x_{k})|^{p}-\E[|\Delta_{i}X(x_{k})|^{p}]\big)\,.
\]
Peligrad and Utev \cite[Thm. B]{PeligradUtev1997} established the central limit theorem $\sum_{i=1}^{n}Z_{n,i}\overset{d}{\to}\mathcal{N}(0,v^{2})$, with variance $v^{2}:=\lim_{n\to\infty}\Var(\sum_{i=1}^{n}Z_{n,i})$, under the following conditions:
\begin{enumerate}
\item[(A)] The variances satisfy $\limsup_{n\to\infty}\sum_{i=1}^{n}\Var(Z_{n,i})<\infty$ and there is a constant $C>0$, such that
\[
\Var\Big(\sum_{i=a}^{b}Z_{n,i}\Big)\le C\sum_{i=a}^{b}\Var(Z_{n,i})\quad\text{for all }0\le a\le b\le n.
\]
\item[(B)] The Lindeberg condition is fulfilled: 
\begin{equation*}
\lim_{n\to\infty}\sum_{i=1}^{n}\E[Z_{n,i}^{2}\1_{\{|Z_{n,i}|>\eps\}}]=0\quad\text{for all }\eps>0.
\end{equation*}
\item[(C)] The following covariance inequality is satisfied. For all $t\in\R$, there is a function $\rho_{t}(u)\ge0,u\in\N,$ satisfying $\sum_{j\ge1}\rho_{t}(2^{j})<\infty$, such that for all integers $1\le a\le b<b+u\le c\le n$:
\[
\Cov(e^{it\sum_{i=a}^{b}Z_{n,i}},e^{it\sum_{i=b+u}^{c}Z_{n,i}})\le\rho_{t}(u)\sum_{i=a}^{c}\Var(Z_{n,i}).
\]
\end{enumerate}
Therefore, Theorem~\ref{thm:clt} follows if the conditions (A) to (C) are verified. (C) is a $\rho$-mixing type condition generalizing the more restrictive condition from \cite{utev1991} that the triangular array is $\rho$-mixing with a certain decay of the mixing coefficients.\\

\noindent\emph{Proof of Theorem~\ref{thm:clt}}
 (A) follows from Proposition~\ref{prop:CovInk}. More precisely, we can verify analogously to the proofs of the Corollaries~\ref{cor:expVar} and \ref{cor:varVnm} that 
\begin{equation*}
\Var(Z_{n,i})=\frac{1}{n}\Big(\Big(\frac{2}{\pi\theta}\Big)^{\frac{p}{2}}\sigma^{2p}\Big((\mu_{2p}-\mu_p^2)+2\sum_{r=2}^{\infty}\rho_p\big(\tfrac12 D_2(\sqrt{\cdot},r)\big)\Big)+\mathcal{O}\Big(\frac{\Delta_{n}}{\delta_{m}^2}\Big)\Big),
\end{equation*}
and we obtain that
\begin{align*}
&\Var\Big(\sum_{i=a}^{b}Z_{n,i}\Big)  =\frac{(b-a+1)}{ n}\Big(\frac{2}{\pi\theta}\Big)^{\frac{p}{2}}\sigma^{2p}\Big((\mu_{2p}-\mu_p^2)+2\sum_{r=2}^{\infty}\rho_p\big(\tfrac12 D_2(\sqrt{\cdot},r)\big)\Big)\\
&\hspace*{7cm}+\mathcal{O}\Big(\frac{(b-a+1)}{ n}\frac{\Delta_{n}}{\delta_{m}^2}+\frac{1}{n}\Big).
\end{align*}
(B) is implied by the Lyapunov condition, since the normal distribution of $\Delta_iX(x_k)$ yields with some constant $C$ that
\[
\sum_{i=1}^{n}\E[Z_{n,i}^{4}]\le C\sum_{i=1}^{n}\big(\E[Z_{n,i}^2]\big)^{2}=\mathcal{O}(n^{-1}) \to0.
\]
(C) Define $Q_{a}^{b}:=\sum_{i=a}^{b}Z_{n,i}$. For a decomposition $Q_{b+u}^{c}:=\sum_{i=b+u}^{c}Z_{n,i}=A_{1}+A_{2}$, where $A_{2}$ is independent of $Q_{a}^{b}$, an elementary estimate with the Cauchy-Schwarz inequality shows that
\begin{equation}
|\Cov(e^{itQ_{a}^{b}},e^{itQ_{b+u}^{c}})|\le2t^{2}\Var(Q_{a}^{b})^{1/2}\Var(A_{1})^{1/2}\,,\label{eq:mixing}
\end{equation}
see \cite[(52)]{BibingerTrabs2017}. To determine such a suitable decomposition, we write for $i>b$
\begin{align}
\Delta_{i}X(x)  &:=\overline B_{i}^b(x)+\underline{B}_{i}^b(x),\qquad\text{where}\nonumber \\
\overline B_{i}^b(x)  &:=\int_{0}^{t_{b}}\int_{\R}\Delta_{i}G(s,x-y)\sigma\dot{W}(ds,dy)\,,\label{eq:ds}\\
 \Delta_{i}G(s,x)&:=G(t_{i}-s,x)-G(t_{i-1}-s,x)\,,\nonumber\\
\underline{B}_{i}^b(x)  &:=\int_{t_{b}}^{t_{i-1}}\int_{\R}\Delta_{i}G(s,x-y)\sigma\dot{W}(ds,dy) \nonumber \\
&\quad +\int_{t_{i-1}}^{t_{i}}\int_{\R}G(t_{i}-s,x-y)\sigma\dot{W}(ds,dy)\,.\nonumber 
\end{align}
Then, we set $A_1:=Q_{b+u}^c-A_2$ and
\[
  A_{2}:=\frac{1}{\sqrt{mn}}\sum_{i=b+u}^{c}\sum_{k=1}^{m}\Delta_{n}^{-\frac{p}{4}}\big(\big|\underline{B}_{i}^b(x_k)\big|^p-\E[|\underline B_i^b(x_k)|^p]\big),\]
where $A_{2}$ is indeed independent from $Q_{a}^{b}$. 
\begin{lemma}
\label{lem:aux}Under the conditions of Theorem \ref{thm:clt}, $\Var(A_{1})=\mathcal{O}(u^{-1/2})$ holds.
\end{lemma}

\noindent This auxiliary lemma is proved in Section~\ref{sec:proofs}.
In combination with $\Var(Q_{b+u}^{c})\ge \varpi\,\frac{(c-b-u+1)}{n}$, with some constant $\varpi>0$, and (\ref{eq:mixing}), we obtain condition (C):
\[
|\Cov(e^{itQ_{a}^{b}},e^{itQ_{b+u}^{c}})|=\mathcal{O}\Big(t^{2}u^{-\frac{1}{4}}\Var(Q_{a}^{c})\Big).
\]
This completes the proof of the central limit theorem for $\sum_{i=1}^{n}Z_{n,i}$ and Theorem \ref{thm:clt}.\hfill$\square$
\clearpage
\section{Remaining proofs\label{sec:proofs}}
In this section, we write $A\lesssim B$ for $A=\mathcal{O}(B)$.
\subsection{Proof of Proposition~\ref{prop:CovInk}}
Since $\Delta_i X(x)=\overline B_i^{i-1}(x)+C_i(x)$, with $\overline B_i^{i-1}(x)$ from \eqref{eq:ds} and
\begin{align}C_i(x)=\int_{t_{i-1}}^{t_{i}}\int_{\R}G(t_{i}-s,x-y)\sigma\dot{W}(ds,dy)\,,\end{align}
with $\overline B_{j}^{j-1}(x)$ and $C_{i}(x)$ centred and independent for $j\le i$, we derive for $j\le i$ that 
\begin{align}
&\Cov(\Delta_{i}X(x),\Delta_{j}X(y)) =\E[\Delta_{i}X(x)\Delta_{j}X(y)]\label{eq:covDecomp} \\
 & =\E[\overline B_{i}^{i-1}(x)\overline B_{j}^{j-1}(y)]+\E[\overline B_{i}^{i-1}(x)C_{j}(y)]\1_{\{i\ne j\}}+\E[C_{i}(x)C_{j}(y)]\1_{\{i=j\}}.\nonumber
\end{align}
Noting that $G(t,\cdot)$ is the density of $\mathcal{N}(0,\theta t)$, we obtain for $x_{1},x_{2}\in\R,r_{1},r_{2}\in(s,\infty)$ based on the identity for the convolution that
\begin{align}
&\int_{\R}G(r_{1}-s,x_{1}-y)G(r_{2}-s,x_{2}-y)dy \label{eq:covInt}\\
&\quad =\int_{\R}G(r_{1}-s,u)G(r_{2}-s,(x_{2}-x_{1})-u)du=G(r_{1}+r_{2}-2s,x_{2}-x_{1}).\nonumber  
\end{align}
We moreover obtain for $r_{3}\le(r_{1}+r_{2})/2$ and $y\ge0$:
\begin{align*}
&\int_{0}^{r_{3}}G(r_{1}+r_{2}-2s,y)ds  =\int_{r_1+r_2-2r_3}^{r_{1}+r_2}\frac12 \frac{1}{\sqrt{2\pi\theta u}}e^{-y^2/(2\theta u)}\,du\\
&=\frac{1}{\sqrt{2\pi\theta}}\big(\sqrt{r_{1}+r_{2}}\,e^{-y^{2}/(2\theta(r_{1}+r_{2}))}-\sqrt{r_{1}+r_{2}-2r_{3}}\,e^{-y^{2}/(2\theta(r_{1}+r_{2}-2r_{3}))}\big)\\
 & \quad-\frac{y}{\theta}\,\P\Big(\frac{y}{\sqrt{r_{1}+r_{2}}}\le\sqrt{\theta}Z\le\frac{y}{\sqrt{r_{1}+r_{2}-2r_{3}}}\Big),\qquad Z\sim\mathcal{N}(0,1).
\end{align*}
Based on that, we determine the terms in \eqref{eq:covDecomp}. Setting
\begin{equation}
\kappa:=|x-y|/\sqrt{\Delta_{n}},\, g_{\kappa}(s):=\sqrt{s}e^{-\kappa^{2}/(2\theta s)}~\text{and}~ h_{\kappa}(s):=\P(Z\ge\kappa/\sqrt{\theta s}),\label{eq:gh}
\end{equation}
we obtain for $j\le i$ by the generalization of It\^{o}'s isometry for Walsh integrals
\begin{align}
&\E[\overline B_{i}^{i-1}(x)\overline B_{j}^{j-1}(y)]\label{eq:bb} =\sigma^{2}\E\Big[\Big(\int_{0}^{t_{i-1}}\int_{\R}\Delta_{i}G(s,x-z)\dot{W}(ds,dz)\Big)
\\ &\hspace*{4cm}\times \Big(\int_{0}^{t_{j-1}}\int_{\R}\Delta_{j}G(s,y-z)\dot{W}(ds,dz)\Big)\Big]\nonumber\\
 & \quad =\sigma^{2}\int_{0}^{t_{j-1}}\int_{\R}\Delta_{i}G(s,x-z)\Delta_{j}G(s,y-z)dzds\nonumber\\
 & \quad=\sigma^{2}\sqrt{\Delta_{n}}\sqrt{\frac{2}{\pi\theta}}\Big(\tfrac12 g_{\kappa}(i+j)-\tfrac12 g_{\kappa}(i-j+2)-g_{\kappa}(i+j-1)\nonumber\\
 & \qquad\qquad+g_{\kappa}(i-j+1)+\tfrac12 g_{\kappa}(i+j-2)-\tfrac12 g_{\kappa}(i-j)\Big)\nonumber\\
 & \qquad-\sigma^{2}\sqrt{\Delta_{n}}\frac{\kappa}{\theta}\Big(h_{\kappa}(i+j)-h_{\kappa}(i-j+2)-2h_{\kappa}(i+j-1)\nonumber\\
 & \qquad\qquad+2h_{\kappa}(i-j+1)+h_{\kappa}(i+j-2)-h_{\kappa}(i-j)\Big)\nonumber\\
 & \quad =\sqrt{\Delta_{n}}\frac{\sigma^{2}}{2}\sqrt{\frac{2}{\pi\theta}}\big( D_{2}(g_{\kappa},i+j)- D_{2}(g_{\kappa},i-j+2)\big)\nonumber\\
 &\qquad\qquad-\sigma^{2}\sqrt{\Delta_{n}}\frac{\kappa}{\theta}\big(D_{2}(h_{\kappa},i+j)-D_{2}(h_{\kappa},i-j+2)\big).\nonumber
\end{align}
Similarly, we have for $j<i$ that
\begin{align}
&\E[\overline B_{i}^{i-1}(x)C_{j}(y)]  =\sigma^{2}\int_{t_{j-1}}^{t_{j}}\int_{\R}\Delta_{i}G(s,x-z)G(t_{j}-s,y-z)dzds\label{eq:bc}\\
 &~ =\sigma^{2}\int_{t_{j-1}}^{t_{j}}\Big(G(t_{i}+t_{j}-2s,x-y)-G(t_{i-1}+t_{j}-2s,x-y)\Big)ds\nonumber \\
 &~ =\frac{\sigma^{2}}{2}\sqrt{\Delta_{n}}\sqrt{\frac{2}{\pi\theta}}\Big( g_{\kappa}(i-j+2)- g_{\kappa}(i-j)-g_{\kappa}(i-j+1)+g_{\kappa}(i-j-1)\Big)\nonumber \\
  &\qquad-\sigma^{2}\sqrt{\Delta_{n}}\frac{\kappa}{\theta}\Big(h_{\kappa}(i-j+2)-h_{\kappa}(i-j)-h_{\kappa}(i-j+1)+h_{\kappa}(i-j-1)\Big).\nonumber 
\end{align}
For $i=j$, with $g_{\kappa}(0)=h_{\kappa}(0)=0$, we obtain that 
\begin{align}
\E[C_{i}(x)C_{i}(y)] & =\sigma^{2}\int_{t_{i-1}}^{t_{i}}\int_{\R}G(t_{i}-s,x-z)G(t_{i}-s,y-z)dzds\label{eq:cc}\\
 & =\frac{\sigma^{2}}{2}\sqrt{\Delta_{n}}\sqrt{\frac{2}{\pi\theta}}g_{\kappa}(2)-\sigma^{2}\sqrt{\Delta_{n}}\frac{\kappa}{\theta} h_{\kappa}(2).\nonumber 
\end{align}
Inserting \eqref{eq:bb}, \eqref{eq:bc} and \eqref{eq:cc} in \eqref{eq:covDecomp} yields
\begin{align*}
 & \Cov(\Delta_{i}X(x),\Delta_{j}X(y))\\
&~=\sigma^{2}\sqrt{\Delta_{n}}\sqrt{\frac{2}{\pi\theta}}\big(g_{\kappa}(1)\1_{\{i=j\}}+\tfrac12 D_{2}(g_{\kappa},|i-j|+1)\1_{\{i\ne j\}}+\tfrac12 D_{2}(g_{\kappa},i+j)\big)\\
 &  \qquad -\sigma^{2}\sqrt{\Delta_{n}}\frac{\kappa}{\theta}\big(2h_{\kappa}(1)\1_{\{i=j\}}+D_{2}(h_{\kappa},|i-j|+1)\1_{\{i\ne j\}}+D_{2}(h_{\kappa},i+j)\big).
\end{align*}
For $x=y$ we have $\kappa=0$ and obtain the result in Proposition \ref{prop:CovInk}. Since the second derivative of $g_{\kappa}$ is bounded by $|g_{\kappa}''(s)|=\frac{1}{4}|s^{-3/2}+2\kappa^{2}\theta^{-1}s^{-5/2}-\kappa^{4}\theta^{-2}s^{-7/2}|e^{-\kappa^{2}/(2\theta s)}\lesssim(\kappa s)^{-1}$ for all $s>0$, we deduce $D_{2}(g_{\kappa},s)\lesssim\kappa^{-1}(s-2)^{-1}$
for $s>2$. Similarly, $|h_{\kappa}''(s)|\lesssim(\kappa s^{-5/2}+\kappa^{3}s^{-7/2})e^{-\kappa^{2}/(2\theta s)}\lesssim\kappa^{-2}s^{-1}$
implies $\kappa D_{2}(h_{\kappa},s)\lesssim\kappa^{-1}(s-2)^{-1}$ for $s>2$. With $g_{\kappa}(s)+h_{\kappa}(s)\lesssim\kappa^{-1}$ for $s\in[0,2]$, we conclude that for $x\neq y$: 
\[
|\Cov(\Delta_{i}X(x),\Delta_{j}X(y))|\lesssim\frac{\Delta_{n}}{|x-y|}\Big(\frac{1}{|i-j-1|\vee1}+\1_{\{i=j\}}\Big).\tag*{{\qed}}
\]
\subsection{Proof of Lemma~\ref{lem:aux}} 
We use that the absolute power functions have an Hermite expansion with Hermite rank 2, that is,
\begin{equation}\label{hermite}|x|^p-\mu_p=\sum_{q\ge 2}a_q\,H_q(x)\,,\end{equation}
with $\mu_p$ from Corollary \ref{cor:expVar} and $H_q$ the $q$th Hermite polynomial and $a_2>0$, see Equation (5.2) of \cite{power}. 
The variance $\Var(A_1)$ coincides with the one of 
\begin{equation*}\tilde A_1=\frac{1}{\sqrt{mn}}\sum_{i=b+u}^c\sum_{k=1}^m\Delta_n^{-\frac{p}{4}}\big(\big|\overline B_{i}^b(x_k)+\underline B_{i}^b(x_k)\big|^p-\big|\underline B_{i}^b(x_k)\big|^p\big)\,,\end{equation*}
where the only difference to $A_1$ is that the expectation is not subtracted. $\overline B_{i}^b(x)$ and $\underline B_{i}^b(x)$ in \eqref{eq:ds} are independent, centred and jointly normally distributed. A first-order Taylor expansion with integral form of the remainder and the relation $H^{\prime}_q(x)=qH_{q-1}(x)$ yields that
\begin{align*}&\Delta_n^{\frac{p}{2}}\Var(\tilde A_1)=\frac{1}{mn}\sum_{i,j=b+u}^c\sum_{k,l=1}^m\Cov\Big(|\underline B_{i}^b(x_k)+\overline B_{i}^b(x_k)|^p-|\underline B_{i}^b(x_k)|^p\,,\\
&\hspace*{8cm}|\underline B_{j}^b(x_l)+\overline B_{j}^b(x_l)|^p-|\underline B_{j}^b(x_l)|^p\Big)\\
& =\frac{1}{mn}\sum_{i,j=b+u}^c\sum_{k,l=1}^m\Cov\bigg(\sum_{r\ge 2}a_r\Big(H_{r}\big(\underline B_{i}^b(x_k)+\overline B_{i}^b(x_k)\big)-H_{r}\big(\underline B_{i}^b(x_k)\big)\Big)\,,\\
&\hspace*{5.5cm}\sum_{s\ge 2}a_s\Big(H_{s}\big(\underline B_{i}^b(x_k)+\overline B_{i}^b(x_k)\big)-H_{s}\big(\underline B_{i}^b(x_k)\big)\Big)\bigg)\\
&=\frac{1}{mn}\sum_{i,j=b+u}^c\sum_{k,l=1}^m\Cov\Big(\sum_{r\ge 2}ra_r\int_0^1\overline B_{i}^b(x_k)H_{r-1}\big(\underline B_{i}^b(x_k)+\tau\overline B_{i}^b(x_k)\big)d\tau\,,\\
&\hspace*{5cm}\sum_{s\ge 2}sa_s\int_0^1\overline B_{j}^b(x_l)H_{s-1}\big(\underline B_{j}^b(x_l)+\tilde\tau\overline B_{j}^b(x_l)\big)d\tilde\tau\Big)\\
&=\frac{1}{mn}\sum_{i,j=b+u}^c\sum_{k,l=1}^m\int_0^1\int_0^1\sum_{r,s\ge 2}rsa_ra_s\overline v_i^k\overline v_j^l\,\Cov\Big(H_1(\overline B_{i}^b(x_k)/\overline v_i^k)H_{r-1}\big(\underline B_{i}^b(x_k)+\tau\overline B_{i}^b(x_k)\big)\,,\\
&\hspace*{7.25cm}H_1(\overline B_{j}^b(x_l)/\overline v_j^l)H_{s-1}\big(\underline B_{j}^b(x_l)+\tilde\tau\overline B_{j}^b(x_l)\big)\Big)\,d\tau d\tilde\tau.
\end{align*}
The variances of $(\underline B_{j}^b(x_l)+\tau\overline B_{j}^b(x_l))$ for all $\tau$ and $\Var(\overline B_{j}^b(x_l))=(\overline v_j^l)^2$ are for all $j,l$ constants multiplied with $\sqrt{\Delta_n}$ and $\sqrt{\Delta_n}(j-b)^{-3/2}$, respectively. Hence, we obtain a factor $\Delta_n^{p/2}$ which cancels out with the factor $\Delta_n^{-p/2}$ and we can transform $(\underline B_{j}^b(x_l)+\tau\overline B_{j}^b(x_l))_{j,l}$ to Gaussian random variables with unit variances where the constant factors by the transformations are not important for our upper bound on the decay in $u$. We can then simplify the sum of covariances using the Isserlis-type moment formula from Lemma 3.2 of Taqqu \cite{taqqu}. For $(X_1,X_2,Y_1,Y_2)^{\top}$ centred multivariate Gaussian such that $(X_i+Y_i)$ have unit variances and where $X_i$ have variances $\sigma_i^2$ and are independent of $Y_j$ for $i,j=1,2$, Taqqu's formula yields for $r,s\ge 2$ with some constants $C_{r,s}^{(a)},1\le a\le 4$, that
\begin{align*}
 & \mathbb{E}[H_{1}(\sigma_1^{-1}X_{1})H_{r-1}(X_{1}+Y_{1})H_{1}(\sigma_2^{-1}X_{2})H_{s-1}(X_{2}+Y_{2})]\\
&\quad =  (\sigma_1\sigma_2)^{-1}\Big(C_{r,s}^{(1)}\1_{\{r=s\}}\mathbb{E}[X_{1}X_{2}]\big(\E[X_{1}X_{2}]+\E[Y_{1}Y_{2}]\big)^{r-1}\\
 & \quad+C_{r,s}^{(2)}\1_{\{r=s\}}\big(\mathbb{E}[X_{1}^{2}]\mathbb{E}[X_{2}^{2}]+\mathbb{E}[X_{1}X_{2}]^{2}\big)\big(\E[X_{1}X_{2}]+\E[Y_{1}Y_{2}]\big)^{r-2}\\
 & \quad+C_{r,s}^{(3)}\1_{\{r=s+2\}}\mathbb{E}[X_{1}X_{2}]\E[X_{1}^{2}]\big(\E[X_{1}X_{2}]+\E[Y_{1}Y_{2}]\big)^{s-1}\\
 & \quad +C_{r,s}^{(4)}\1_{\{s=r+2\}}\E[X_{1}X_{2}]\mathbb{E}[X_{2}^{2}]\big(\E[X_{1}X_{2}]+\E[Y_{1}Y_{2}]\big)^{r-1}\Big)\,.
\end{align*}
By Taqqu's formula most cross terms with $r\ne s$ vanish in the identity above. Only terms with $|r-s|=2$ yield some non-vanishing summands. Except for some summands in the second line of the last equality, all other summands include either a factor $\E[X_1X_2]^2$ or $\E[X_1X_2]\E[Y_1Y_2]$. From the remaining summands with $\E[X_{1}^{2}]\E[X_{2}^{2}]\E[Y_1Y_2]^{r-2}$, using that $\E[H_u(X_1)H_v(X_1+Y_1)]=\1_{\{u=v\}}u!\E[X_1^2]$, the summand for $r=2$ cancels out in the covariance. Hence, it suffices to consider these three types of summands, the last only for $r\ge 4$ as coefficients for odd $r$ vanish. In order to derive an upper bound for $\Var(\tilde A_{1})$, we thus determine $\E[\overline B_{i}^b(x)\overline B_{j}^b(y)]$ and $\E[\underline B_{i}^b(x)\underline B_{j}^b(y)]$. To evaluate these terms, we conduct similar calculations as in the proof of Proposition~\ref{prop:CovInk}. For $b\le j\le i$, any $x,y\in\R$ and with the notation from \eqref{eq:gh}:
\begin{align*}
\E\big[\overline B_{i}^b(x)\overline B_{j}^b(y)\big] & =\sigma^{2}\int_{0}^{t_{b}}\int_{\R}\Delta_{i}G(s,x-z)\Delta_{j}G(s,y-z)dzds\\
 & =\frac{\sigma^{2}}{2}\sqrt{\Delta_{n}}\sqrt{\frac{2}{\pi\theta}}\big(D_{2}(g_{\kappa},i+j)-D_{2}(g_{\kappa},i+j-2b)\big)\\
 & \quad-\sigma^{2}\sqrt{\Delta_{n}}\frac{\kappa}{\theta}\big(D_{2}(h_{\kappa},i+j)-D_{2}(h_{\kappa},i+j-2b)\big).
\end{align*}
Since $|D_{2}(g_{0},s)|\lesssim s^{-3/2}$ and $|D_{2}(g_{\kappa},s)|+|\kappa D_{2}(h_{\kappa},s)|\lesssim\kappa^{-1}(s-2)^{-1}$ for $s>2$ and $\kappa>0$ as shown at the end of the proof of Proposition~\ref{prop:CovInk}, we conclude that
\begin{equation}
\big|\E\big[\overline B_{i}^b(x)\overline B_{j}^b(y)\big]\big|\lesssim\sqrt{\Delta_{n}}(i+j-2b)^{-\frac32}\1_{\{x=y\}}+\frac{\Delta_{n}}{|x-y|}(i+j-2b-2)^{-1}\1_{\{x\neq y\}}.\label{eq:d1d1}
\end{equation}
To bound $\E[\underline B_{i}^b(x)\underline B_{j}^b(y)]$, we use for $j\le i$ that 
\begin{align*}
\E[\underline B_{i}^b(x)\underline B_{j}^b(y)]= & \sigma^{2}\int_{t_{b}}^{t_{j-1}}\int_{\R}\Delta_{i}G(s,x-z)\Delta_{j}G(s,y-z)dzds\\
 & \quad +\1_{\{i\ne j\}}\E[\overline B_{i}^{i-1}(x)C_{j}(y)]+\1_{\{i=j\}}\E[C_{i}(x)C_{j}(y)]\,.
\end{align*}
The second and third summand have already been determined in the proof of Proposition~\ref{prop:CovInk}. For the first one, we obtain that
\begin{align*}
&\sigma^{2}\hspace*{-.05cm}\int_{t_{b}}^{t_{j-1}}\hspace*{-.15cm}\int_{\R}\Delta_{i}G(s,x-z)\Delta_{j}G(s,y-z)dzds\\
&\quad =\frac{\sigma^{2}\Delta_{n}^{1/2}}{2}\sqrt{\frac{2}{\pi\theta}}\big(D_{2}(g_{\kappa},i+j-2b)\hspace*{-.05cm}-\hspace*{-.05cm}D_{2}(g_{\kappa},i-j+2)\big)\\
&\qquad -\sigma^{2}\sqrt{\Delta_{n}}\frac{\kappa}{\theta}\big(D_{2}(h_{\kappa},i+j-2b)-D_{2}(h_{\kappa},i-j+2)\big)\,.
\end{align*}
Inserting the three summands, we derive that
\begin{align*}
\E[\underline B_{i}^b(x)\underline B_{j}^b(y)]  &=\frac{\sigma^{2}\Delta_{n}^{1/2}}{2}\sqrt{\frac{2}{\pi\theta}}\big(D_{2}(g_{\kappa},i+j-2b)+D_{2}(g_{\kappa},|i-j|+1)\1_{\{i\ne j\}}+2g_{\kappa}(1)\1_{\{i=j\}}\big)\\
 & \hspace*{1.1cm}-\sigma^{2}\sqrt{\Delta_{n}}\frac{\kappa}{\theta}\big(D_{2}(h_{\kappa},i+j-2b)+D_{2}(h_{\kappa},|i-j|+1)\1_{\{i\ne j\}}+2h_{\kappa}(1)\1_{\{i=j\}}\big)\\
 & \lesssim\Delta_{n}^{1/2}\Big(\frac{1}{|i-j-1|\vee1}+\1_{\{i=j\}}\Big)\Big(\1_{\{x=y\}}+\frac{\Delta_{n}^{1/2}}{|x-y|}\1_{\{x\neq y\}}\Big).
\end{align*}
Consider first the sum of covariances including the factors $\Delta_n^{-1}\E[\overline B_{i}^b(x_k)\overline B_{j}^b(x_l)]^2$. We obtain with \eqref{eq:d1d1} the bound
\begin{align*}
&\frac{C_p}{n}\sum_{i,j=b+u}^{c}(i+j-2b)^{-3}+\frac{C_p}{mn}\sum_{i,j=b+u}^{c}\sum_{k\ne l}\frac{\Delta_n}{|x_{k}-x_{l}|^2}(i+j-2b-2)^{-2}\\
&\lesssim \frac{c-b-u+1}{n}\sum_{k\ge u}k^{-3}+\frac{\Delta_{n}}{\delta_{m}^2}\frac{c-b-u+1}{n}\sum_{l\ge 1}l^{-2}\sum_{k\ge u}k^{-2}\lesssim u^{-1}\,,
\end{align*}
where the last step is similar to \eqref{h_spatial}.\\ For the terms with the factors $\Delta_n^{-1}\E[\underline B_{i}^b(x_k)\underline B_{j}^b(x_l)]\E[\overline B_{i}^b(x_k)\overline B_{j}^b(x_l)]$, we obtain that
\begin{align*}
&\frac{C_p}{n}\sum_{i,j=b+u}^{c}(i+j-2b)^{-\frac32}\Big(\frac{1}{|i-j-1|\vee1}+\1_{\{i=j\}}\Big)\\
 & \;+\frac{C_p}{mn}\sum_{i,j=b+u}^{c}\sum_{k\neq l}\frac{\Delta_{n}}{|x_{k}-x_{l}|^{2}}(i+j-2b-2)^{-1}\Big(\frac{1}{|i-j-1|\vee1}+\1_{\{i=j\}}\Big)\\
& \lesssim\frac{(c-b-u+1)}{n}\sum_{k\ge u} k^{-\frac32}+\frac{(c-b-u+1)}{n}\frac{\Delta_n}{\delta_m^2}\frac{1}{\sqrt{u}}\sum_{k\ge 1} k^{-\frac32}\\
&\lesssim u^{-\frac12} \frac{(c-b-u+1)}{n}\Big(1+\frac{\Delta_{n}}{\delta_{m}^{2}}\Big)\lesssim u^{-\frac12}\,.
\end{align*}
The terms with factors $\Delta_n^{-2}\Var(\overline B_{i}^b(x_k))\Var(\overline B_{j}^b(x_l))\E[\underline B_{i}^b(x_k)\underline B_{j}^b(x_l)]^2$ are bounded by
\begin{align*}
&\frac{C_p}{n}\sum_{i,j=b+u}^{c}\Big(\frac{1}{|i-j-1|\vee1}+\1_{\{i=j\}}\Big)^2(i-b)^{-\frac32}(j-b)^{-\frac32}\\
 & \;+\frac{C_p}{mn}\sum_{i,j=b+u}^{c}\sum_{k\neq l}\frac{\Delta_{n}}{|x_{k}-x_{l}|^2}\Big(\frac{1}{|i-j-1|\vee1}+\1_{\{i=j\}}\Big)^2(i-b)^{-\frac32}(j-b)^{-\frac32}\\
& \lesssim\frac{(c-b-u+1)}{n}u^{-\frac12}\sum_{k\ge 1} k^{-\frac52}+\frac{(c-b-u+1)}{n}u^{-\frac12}\frac{\Delta_n}{\delta_m^2}\sum_{l\ge 1}l^{-2}\sum_{k\ge 1} k^{-\frac72}\\
& \lesssim u^{-\frac12}\,.
\end{align*}
Inserting these bounds in the identity for $\Var(\tilde A_1)$, yields that $\Var(A_1)\lesssim u^{-1/2}$. \qed

\vspace*{.5cm}

\end{document}